\documentclass[11pt]{amsart}
\usepackage{hyperref}

\usepackage{enumerate}

\usepackage[utf8]{inputenc}
\usepackage[T1]{fontenc}

\usepackage[mathscr]{eucal}

\usepackage[a4paper, twoside=false, vmargin={2cm,3cm}, includehead]{geometry}

\newtheorem{lemma}{Lemma}[section]
\newtheorem{theorem}[lemma]{Theorem}
\newtheorem*{theorem*}{Theorem}
\newtheorem{corollary}[lemma]{Corollary}

\newtheorem{proposition}[lemma]{Proposition}
\newtheorem*{proposition*}{Proposition}

\theoremstyle{definition}
\newtheorem{definition}{Definition}
\newtheorem*{claim*}{Claim}

\newtheorem*{remark}{Remark}
\newtheorem*{remarks}{Remarks}

\newcommand{\D}{{\mathbb D}}
\newcommand{\C}{{\mathbb C}}
\newcommand{\E}{{\mathbb E}}

\newcommand{\N}{{\mathbb N}}
\renewcommand{\P}{{\mathbb P}}
\newcommand{\Q}{{\mathbb Q}}
\newcommand{\R}{{\mathbb R}}
\newcommand{\T}{{\mathbb T}}
\newcommand{\Z}{{\mathbb Z}}

\newcommand{\CM}{{\mathcal M}}

\newcommand{\CX}{{\mathcal X}}

\newcommand{\bn}{{\mathbf{n}}}

\newcommand{\one}{\mathbf{1}}

\newcommand{\norm}[1]{\left\Vert #1\right\Vert}

\DeclareMathOperator{\UD}{UD-lim}
\DeclareMathOperator{\CD}{D-lim}

\DeclareMathOperator{\reel}{Re}
\renewcommand{\Re}{\reel}

\begin{document}

\title[An averaged Chowla-Elliott conjecture
along  independent polynomials]{An averaged Chowla and Elliott  conjecture
along  independent polynomials}

\author{Nikos Frantzikinakis}
\address[Nikos Frantzikinakis]{University of Crete, Department of Mathematics and Applied Mathematics, Voutes University Campus, Heraklion 71003, Greece} \email{frantzikinakis@gmail.com}

\begin{abstract}
We generalize a result of  Matom\"aki,  Radziwi{\l}{\l}, and Tao, by proving an  averaged version of  a conjecture of Chowla and a conjecture of  Elliott regarding correlations of the Liouville function, or more general bounded multiplicative functions, with shifts given by independent polynomials in several variables.
A new feature is that we recast the problem in ergodic terms and   use a  multiple ergodic theorem
to prove it;
its hypothesis is verified using
recent results by  Matom\"aki and  Radziwi{\l}{\l} on mean values of multiplicative functions on typical short intervals.
We deduce several consequences about patterns that  can be found on the range of  various arithmetic sequences along shifts of independent polynomials.
\end{abstract}





\maketitle
\section{Introduction}
Let $\lambda\colon \N\to \{-1,+1\}$ be the Liouville function which  is defined to be $1$ on integers with an even number of prime factors, counted with  multiplicity, and $-1$ elsewhere.
A well known conjecture of Chowla \cite{Ch65} asserts that if  $n_1,\ldots, n_\ell\in \N$ are distinct, then
$$
\lim_{M\to \infty} \frac{1}{M}\sum_{m=1}^M\lambda(m) \, \lambda(m+n_1)\cdots \lambda(m+n_\ell)=0.
$$
The conjecture remains open even when  $\ell=1$. Very recently, a  version involving logarithmic averages
was established for $\ell=1$ in \cite{Tao15} and an averaged form of Chowla's conjecture
was established in \cite{MRT15}. The latter implies that if $(M_k)_{k\in\N}$ is a subsequence of the positive integers such that for  every $\bn=(n_1,\ldots, n_\ell)\in \N^\ell$ the  correlation limit
$$
\text{C}_{\text{lin}}(\bn):=\lim_{k\to\infty}\frac{1}{M_k}\sum_{m=1}^{M_k}\lambda(m) \, \lambda(m+n_1)\cdots \lambda(m+n_\ell)
$$
 exists, then the sequence
$(\text{C}_{\text{lin}}(\bn))$ converges to $0$ in uniform density (see Definition~\ref{D:UD-lim}).
We  extend this result to the case where the linear polynomials $n_1,\ldots, n_\ell$ are replaced by any collection of  independent polynomials $p_1,\ldots, p_\ell \colon \N^r \to \Z$ where $\ell, r\in \N$.
In particular, we show that the sequence $(\text{C}_\text{pol}(\bn))$,  defined by
$$
\text{C}_{\text{pol}}(\bn):=\lim_{k\to\infty}\frac{1}{M_k}\sum_{m=1}^{M_k}\lambda(m) \, \lambda(m+p_1(\bn))
\cdots \lambda(m+p_\ell(\bn)), \quad \bn\in \N^r,
$$
converges to $0$ in uniform density. A particular case of interest is when $r=1$ and $p_j(n)=n^j$ for  $j=1,\ldots,\ell$.

More generally, Elliott conjectured \cite[Conjecture~II]{El90,El94} that  if the multiplicative functions $f_0,\ldots, f_\ell$ take values on the complex unit disc and are aperiodic (see Definition~\ref{D:aperiodic}), then
for all distinct   $n_1,\ldots, n_\ell\in \N$ we have
$$
\lim_{M\to \infty} \frac{1}{M}\sum_{m=1}^Mf_0(m) \, f_1(m+n_1)\cdots f_\ell(m+n_\ell)=0.
$$
In \cite[Theorem~B.1]{MRT15} it was shown that in this generality the conjecture is false on a technicality and
a  modification of the conjecture was suggested where
 one works under the somewhat more restrictive assumption of strong aperiodicity (see Definition~\ref{D:uniformly}).  An averaged form of the
modified conjecture was proved in    \cite[Theorem~1.6]{MRT15} and we give a polynomial variant of this result in Theorem~\ref{T:main}.

As was the case in  \cite{MRT15}, where correlations along linear shifts
  were studied,  our argument
is based on a recent result from \cite{MR15} on averages of multiplicative functions on typical short intervals.   But due to the polynomial nature of our problem,
a serious additional difficulty appears, since Fourier transform techniques (and their higher order variants) are not suited for the study of  polynomial shifts.\footnote{Moreover, the recent concatenation techniques for Gowers norms from \cite{TZ16a,TZ16b}  are devised to treat  different
averages that are taken jointly on  the parameters $m$ and $\bn$;
 they also apply  to  different classes of polynomials (see Section~\ref{SS:further}).} To  overcome this obstacle, we recast the  problem in ergodic terms. Using a
 rather
  deep rooted result  from ergodic theory (Theorem~\ref{T:PolyErgodic}) we deduce that the polynomial correlation sequences  are controlled by some simpler linear ones 
 which  can
be handled using  number theoretic tools like those used in  \cite{MRT15}.



We  use the previous results in order to deduce  several consequences about patterns that can be found on the range of  various arithmetic sequences along shifts of independent polynomials.
 For instance, we show that:

$\bullet$ For every strongly aperiodic  multiplicative function $f$ with values in $\{-1,+1\}$,
 for all $n\in \N$ outside  a set of
 natural density $0$, each  sign pattern  of size $\ell+1$ is taken by $f$ along  progressions of the form $m, m+n, m+n^2, \ldots, m+n^\ell$, and  in fact, with  natural density (with respect to  the variable $m$) that converges to  $2^{-(\ell+1)}$ as $n\to \infty$.

\smallskip

 $\bullet$   For all $b_i\in \N$, $a_i\in \{0,\ldots, b_i-1\}$,  $i=0,\ldots, \ell$, and  all $n\in \N$ outside a set of
 natural  density $0$, the set of $m\in \N$ for which the integers  $m, m+n, m+n^2, \ldots, m+n^\ell$, have $a_0\! \! \mod{b_0}, \ldots, a_\ell\! \! \mod{b_\ell}$  prime factors respectively  is non-empty,  and in fact itsnatural  density
converges to  $(\prod_{j=0}^\ell b_j)^{-1}$ as $n\to \infty$.


In the next section,  we give the precise statements of the results alluded to in  the previous discussion and
also give several  relevant  refinements and open problems.

\section{Main results}
\subsection{Definitions and notation}\label{SS:aperiodic}
In order to  facilitate exposition, we introduce some definitions and notation.
Throughout this article, for  $N\in\N$  we let  $[N]:=\{1,\dots,N\}$.
If $A$ is a finite non-empty set, we let $\E_{a\in A}:=\frac{1}{|A|}\sum_{a\in A}$.
Furthermore, with $\P$ we denote the set of prime numbers.

\begin{definition}\label{D:aperiodic} A function $f\colon \N\to \C$ is called \emph{multiplicative} if
 $$
 f(mn)=f(m)f(n) \ \text{ whenever } \  (m,n)=1.
 $$
 It is called \emph{completely multiplicative} if  the previous identity holds for every $m,n\in \N$.

For convenience,  we extend all multiplicative functions to $\Z$ by letting $f(-n)=f(n)$ and $f(0)=0$,
and  let
$$
\CM:=\{f\colon \Z\to \C \text{ multiplicative such that  } |f(n)|\leq 1 \text{ for every } n\in \Z\}.
$$

A \emph{Dirichlet character}, typically denoted with $\chi$,  is a periodic completely multiplicative function such that $\chi(1)=1$.

 We say that   $f\in \CM$ is \emph{aperiodic} 
 if
$\lim_{N\to\infty}\E_{n\in [N]}f(an+b)=0$  for every $a,b\in \N$, equivalently, if
$\lim_{N\to\infty}\E_{n\in [N]}f(n)\, \chi(n)=0$ for every Dirichlet character $\chi$.
\end{definition}
We define the distance between two multiplicative functions as in  \cite{GS07,GS16}:
\begin{definition}\label{D:D}
If $f,g\in \CM$ we let $\D\colon \CM\times \CM\to [0,\infty]$ be given by
$$
\D(f,g)^2:=\sum_{p\in \P} \frac{1}{p}\,\bigl(1-\Re\bigl(f(p) \overline{g(p)}\bigr)\bigr).
$$
We also let $\D\colon \CM\times \CM\times \N \to [0,\infty)$ be given by
$$
\D(f,g;N)^2:=\sum_{p\in \P\cap [N]} \frac{1}{p}\,\bigl(1-\Re\bigl(f(p) \overline{g(p)}\bigr)\bigr)
$$
and $M\colon \CM\times \N \to [0,\infty)$ be given by
$
M(f;N):=\min_{|t|\leq N} \D(f, n^{it};N)^2.
$
\end{definition}

A celebrated  theorem of Hal\'asz~\cite{Hal68} states that
if   $f\in \CM$, then it  has zero  mean value  if and only if
for every $t\in \R$  we either have $\D(f, n^{it})=\infty$
or
$
f(2^k)= -2^{ikt}$ for all $k\in \N$.

 For our purposes, we need information on averages of multiplicative functions taken on typical short intervals. One such result is  Theorem~\ref{T:MRT} below and its  assumption motivates the following definition:
\begin{definition}\label{D:uniformly}
The multiplicative function $f\in \CM$ is
{\em strongly  aperiodic}
 if $M(f\cdot\chi;N)\to \infty$ as $N
\to \infty$ for every Dirichlet character $\chi$.
\end{definition}
Note  that strong aperiodicity implies aperiodicity. The converse is
not in general true (see  \cite[Theorem~B.1]{MRT15}), but it is
true for  real valued  multiplicative functions (see  \cite[Appendix~C]{MRT15}). In particular,  the  Liouville and the M\"obius function are strongly aperiodic.
More examples of strongly aperiodic multiplicative functions are given in  Corollary~\ref{C:aperiodic}.

A subset $Z$ of $\N^r$ has
 {\em natural density $0$} if
$\lim_{N\to\infty}\frac{|Z\cap [N]^r|}{N^r}=0$, and
{\em  uniform density $0$} if
$\lim_{|I|\to\infty}\frac{|Z\cap I|}{|I|}=0$, where
the last limit is taken over all parallelepipeds in $\N^r$ with side lengths tending to $\infty$.
We make frequent use of the notion of convergence in uniform density which is defined as follows:
\begin{definition}\label{D:UD-lim}
Let $r\in\N$. We say that the sequence $a\in \ell^\infty(\N^r)$ {\em converges in uniform density} to a constant
$c\in \C$,
and write
$\UD_{\bn\to\infty} a(\bn)=c$, if any of the following two equivalent conditions hold:
\begin{enumerate}

\item For every $\varepsilon>0$ the set $\{
\bn\in \N^r\colon |a(\bn)-c|\geq \varepsilon\}$ has uniform density $0$;

\item $\lim_{N\to \infty} \E_{\bn\in I_N}|a(\bn)-c|=0$ for every sequence of parallelepipeds
$(I_N)_{N\in \N}$ in $\N^r$ with side lengths tending to $\infty$.
\end{enumerate}
\end{definition}
\begin{remark}
If we replace the uniform density  with the natural density, then the second condition becomes
$\lim_{N\to \infty} \E_{\bn\in [N]^r}|a(\bn)-c|=0$, and we can add a third equivalent condition, namely,
  $\lim_{|\bn|\to \infty, \bn\notin Z\, }a(\bn)=c$ for some $Z\subset \N^r$ with natural density $0$.
\end{remark}

Lastly, if  ${\bf M}:=([M_k])_{k\in\N}$ is a  sequence of intervals with $M_k\to \infty$ and $a\in \ell^\infty(\N)$, we let
$
\E_{m\in {\bf M}}\, a(m):=\lim_{k\to\infty} \E_{m\in [M_k]}\, a(m)
$
assuming that the limit exists.

\subsection{Averaged Chowla-Elliot conjecture along independent sequences}
We say that the polynomials  $p_1,\ldots, p_\ell \colon \N^r\to \Z$ are \emph{independent}  if
the set $\{1,p_1,\ldots, p_\ell\}$ is linearly independent.
Our main result is the following:
\begin{theorem}\label{T:main}
Let  $p_1,\ldots, p_\ell \colon \N^r\to \Z$ be  independent polynomials and
$f_0,\ldots, f_\ell\in \CM$ be multiplicative functions  at least one of which is strongly aperiodic.
 Then there exists a sequence of intervals ${\bf M}:=([M_k])_{k\in\N}$ with $M_k\to\infty$ such that
\begin{equation}\label{E:main0}
\UD_{\bn\to\infty}  \big(\E_{m\in {\bf M}}\,  f_0(m) \, f_1(m+p_1(\bn))\cdots f_\ell(m+p_\ell(\bn))\big)=0.
\end{equation}
\end{theorem}
\begin{remarks}
$\bullet$ It suffices to choose $(M_k)_{k\in\N}$ so that
  the multiplicative functions $f_0,\ldots, f_\ell\in \CM$  admit correlations  along the sequence of   intervals ${\bf M}$ (see Definition~\ref{D:correlations}). Hence, we can choose $(M_k)_{k\in\N}$ to be a subsequence of any strictly increasing sequence of integers.

$\bullet$ Theorem~1.6 in \cite{MRT15} establishes (a quantitative variant of) this result  for the polynomials  $p_1(\bn)=n_1,\ldots, p_\ell(\bn)=n_\ell$. 
As noted in the introduction,
in order to establish  the polynomial extension  we lend tools from
ergodic theory, in particular, we use  Theorem~\ref{T:PolyErgodic} below.

$\bullet$ It is not hard to modify our argument in order to prove a similar conclusion for averages of the form
$\E_{m\in {\bf M}}\,  f_0(a_0m+c_0) \, f_1(a_1m+p_1(\bn))\cdots f_\ell(a_\ell m+p_\ell(\bn))$ where the polynomials $p_1,\ldots,p_\ell$ are independent,  $a_0,\ldots, a_\ell\in \Z$ are
non-zero,  and $c_0\in \Z$.
\end{remarks}

In a similar fashion, using Theorem~\ref{T:HardyErgodic} as our ergodic input we prove the following:
\begin{theorem}\label{T:Hardy}
Let  $c_1,\ldots, c_\ell\in \R^+$ be distinct non-integers and
 $f_0,\ldots, f_\ell\in \CM$.
There exists a sequence of intervals ${\bf M}:=([M_k])_{k\in\N}$ with $M_k\to\infty$  such that:
\begin{enumerate}
\item If  for some
$j\in \{0,\ldots, \ell\}$ we have  $M(f_j;N)\to \infty$ as $N\to \infty$, then
$$
\lim_{N\to \infty}\E_{n\in [N]}
\E_{m\in {\bf M}}\,  f_0(m) \, f_1(m+[n^{c_1}]))\cdots f_\ell(m+[n^{c_\ell}]) =0.
$$

\item  If  for some
$j\in \{0,\ldots, \ell\}$ we have that $f_j$ is strongly aperiodic, then
$$
\lim_{N\to \infty}\E_{n\in [N]}
 |\E_{m\in {\bf M}}\,  f_0(m) \, f_1(m+[n^{c_1}]))\cdots f_\ell(m+[n^{c_\ell}])|=0.
$$
\end{enumerate}
 \end{theorem}
 \begin{remarks}
 $\bullet$ It is not hard to modify the proof  of part $(ii)$ in order to allow some of the  exponents
 $c_1,\ldots, c_\ell$ to be integers.

 $\bullet$ Using the  ergodic theorems \cite[Theorem~2.6]{Fr10} and \cite[Theorem~2.3]{Fr15} one can substitute the sequences $n^{c_1},\ldots, n^{c_\ell}$
 in the previous statement with collections of sequences having different growth rates  taken  from a much larger class of Hardy field sequences.
\end{remarks}


\subsection{Patterns in arithmetic sequences}
Results regarding sign patterns of arithmetic sequences and in particular of
the Liouville function $\lambda$ or the M\"obius  function $\mu$ have a rich history.
 For instance,
 it is known  that (consecutive) values of  $\lambda$
 attain all possible sign patterns with size two \cite{HPW85}  and size three
 \cite{Hi86}  infinitely often, and in fact with positive lower density \cite{MRT15b}. A similar result is known for $\mu$ but only for patterns of  size two  \cite{MRT15b}. On the other hand, it is not  known that four consecutive ones are taken
 by  $\lambda$
 infinitely often.
 A partial result  for patterns of longer size is that $\lambda$  takes at least
 $k+5$ of the possible $2^k$ sign patterns of size $k$
  with positive upper density \cite[Proposition~2.9]{MRT15b}. Of course,  Chowla's conjecture predicts that $\lambda$
   takes all possible sign patterns of size $k$  with natural density $2^{-k}$.

 On a different direction, it is known that
 $\lambda$ takes infinitely often all sixteen sign patterns along four
 term arithmetic progressions
 $m, m+n, m+2n, m+3n$ where $m\in \N$ and $n$ belongs to  a finite set \cite{BE11},
  and $\lambda$  takes all possible $2^{\ell+1}$ sign patterns along arithmetic
  progressions   $m, m+n,\ldots, m+\ell n$;  each pattern for a proportion of  $(m,n)\in [N]\times [N]$ that converges to $2^{-(\ell+1)}$ as  $N\to \infty$
 \cite[Proposition~9.1]{GT10}. Using \cite[Theorem~1.1]{FH16} we get similar results for any aperiodic multiplicative function with values in $\{-1,+1\}$.

Using Theorem~\ref{T:main} we get  consequences about sign patterns of $\lambda$ and $\mu$ of arbitrary size along
 polynomial progressions given by independent polynomials. To state our results it is convenient to introduce some notation.
Given a sequence of intervals ${\bf M}:=([M_k])_{k\in\N}$ with $M_k\to \infty$ we denote with $d_{\bf M}$  the natural density induced by the sequence of intervals ${\bf M}$, that is, for $\Lambda \subset \N$ we let
 $d_{\bf M}(\Lambda):=\lim_{k\to\infty} \frac{|\Lambda\cap [M_k]|}{M_k}$, assuming that the limit exists. Moreover, we define the upper density of a set $\Lambda \subset \N$ as
  $\overline{d}(\Lambda):=\limsup_{M\to\infty} \frac{|\Lambda\cap [M]|}{M}$.
\begin{theorem}\label{T:sign1}
Let  $p_1,\ldots, p_\ell \colon \N^r\to \Z$ be  independent polynomials and
 $f_0, \ldots, f_\ell\colon \Z\to \{-1,+1\}$ be strongly  aperiodic multiplicative functions.
Furthermore, for $\bn\in \N^r$ and $\epsilon_0,\ldots, \epsilon_\ell\in \{-1,+1\}$,
let ${\bf \epsilon}:=(\epsilon_0,\ldots, \epsilon_\ell)$ and
$$
\Lambda_{\bn,{\bf \epsilon}}:=\{m\in \N \colon f_0(m)=\epsilon_0, f_1(m+p_1(\bn))=\epsilon_1,\ldots, f_\ell(m+p_\ell(\bn))=\epsilon_\ell\}.$$
Then there exists a sequence of intervals ${\bf M}:=([M_k])_{k\in\N}$ with $M_k\to\infty$ such that 
$$
\UD_{\bn\to\infty}(d_{\bf M}(\Lambda_{\bn,{\bf \epsilon}}))=2^{-(\ell+1)}.
$$
\end{theorem}
\begin{remark}
It follows that for every $\varepsilon>0$ we have $
\overline{d}(\Lambda_{\bn,{\bf \epsilon}})
\geq 2^{-(\ell+1)}-\varepsilon$ outside   a set of $\bn\in \N^r$ of uniform density $0$.
\end{remark}
Note that even   when $f_0=\cdots=f_\ell=\lambda$, $r=1$, and $p_j(n)=n^j$, $j=1,\ldots, \ell$,  it is non-trivial to establish that   the set
$\Lambda_{\bn,{\bf \epsilon}}$ is non-empty for all choices of sign patterns.

Given $a,b\in \N,$ with $[a]_b$ we denote the unique integer in $\{0,\ldots, b-1\}$ which is congruent to $a$ $\! \! \! \mod{b}$.
We denote with $\omega(n)$ the number of distinct prime factors of an integer $n$ and  with $\Omega(n)$ the number of prime factors of $n$ counted with  multiplicity.

\begin{theorem}\label{T:sign2}
Let  $p_1,\ldots, p_\ell \colon \N^r\to \Z$ be independent polynomials and
  $b_0,\ldots, b_\ell\in \N$. For  $\bn\in \N^r$ and $a_j\in \{0,..,b_j-1\}$, $j=0,\ldots, \ell$,
let ${\bf a}:=(a_0,\ldots, a_\ell)$ and
$$
\Lambda_{\bn,{\bf a}}:=\{m\in \N \colon [\omega(m)]_{b_0}=a_0 , [\omega(m+p_1(\bn))]_{b_1}=a_1,
\ldots, [\omega(m+p_\ell(\bn))]_{b_\ell}=a_\ell\}.$$
Then there exists a sequence of intervals ${\bf M}:=([M_k])_{k\in\N}$ with $M_k\to\infty$ such that
$$
\UD_{\bn\to\infty}(d_{\bf M}(\Lambda_{\bn,{\bf a}}))=\big(\prod_{j=0}^\ell b_j\big)^{-1}.
$$
\end{theorem}
\begin{remark}
It follows that  for every $\varepsilon>0$ we have $
\overline{d}(\Lambda_{\bn,{\bf a}})\geq(\prod_{j=0}^\ell b_j)^{-1}-\varepsilon$ outside   a set of $\bn\in \N^r$ of uniform density $0$.
\end{remark}


With minor modifications
one can prove variants of the previous  statement where
in the definition of the set
$\Lambda_{\bn,{\bf a}}$
one uses the arithmetic function
 $\Omega$ in the place of $\omega$ or uses
$\omega$  and $\Omega$ in different places.

Moreover,  using part $(ii)$ of Theorem~\ref{T:Hardy} in place of Theorem~\ref{T:main}, one can get variants of the
previous results (with essentially the same proof) where the independent polynomials are
 substituted by sequences given by the  integer part of
different fractional powers and the limit in uniform density is substituted by
 the  limit in natural density, defined by $\CD_{n\to\infty}(a(n))=c$ if $\lim_{N\to\infty}\E_{n\in [N]}|a(n)-c|=0$.

\subsection{Further remarks and open problems}\label{SS:further}
It is natural to ask for   extensions of  Theorem~\ref{T:main} which  cover families of polynomials that are   dependent and have pairwise non-constant differences. In this direction, even showing that
$$
\lim_{N\to\infty} \E_{n\in [N]} \big( \E_{m\in {\bf M}}\,  \lambda(m) \, \lambda(m+n)\,\lambda(m+2n)\big)=0
$$
for some sequence of intervals ${\bf M}:=([M_k])_{k\in\N}$ with $M_k\to\infty$
remains a challenge. In order to have access to  ergodic methodology
one needs to prove  a variant of Proposition~\ref{P:key} below with an assumption on $\lambda$
that enables to deduce orthogonality of the function $F=F_\lambda$ to
a factor (the $\mathcal{Z}_1$-factor)
which is in general    larger than the Kronecker factor of the corresponding system.
The number theoretic input on $\lambda$  needed  to establish such claim is the following asymptotic
\begin{equation}\label{E:conj}
\lim_{N\to \infty} \limsup_{M\to \infty}\E_{m\in [M]} \sup_{t\in \R}| \E_{n\in [N]} \lambda(m+n) \cdot e(nt)|=0.
\end{equation}
This is not known yet, and the reader can consult \cite{Tao15,Tao16} for further consequences in case \eqref{E:conj} holds (it implies the logarithmically averaged Chowla conjecture for three point correlations). Extending  Theorem~\ref{T:main} to general families of polynomials with pairwise non-constant differences would require a higher order variant of \eqref{E:conj} with nilsequences of bounded complexity used in  place of the linear exponential sequences.

In a rather different direction, one can study correlation sequences by taking averages simultaneously on the parameters $m$ and $\bn$. For instance,
 it was shown in \cite[Proposition~9.1]{GT10}, modulo conjectures which  were subsequently verified in \cite{GT12b, GTZ12c}, that
$$
\lim_{N\to\infty}\E_{\bn\in [N]^r} \E_{m\in [N]} \, f(m)\, f(m+p_1(\bn)) \cdots f(m+p_\ell(\bn))=0,
$$
where $f$ is the M\"obius or the Liouville function  and $p_1,\ldots, p_\ell$ are linear polynomials with   pairwise non-constant differences.  In \cite{FH16, FH14} this is extended to the case where $f$ is an arbitrary aperiodic multiplicative function.

 Moreover, the concatenation results from \cite{TZ16a, TZ16b}, that allow to control local Gowers norms by global ones, can be used to show that
\begin{equation}\label{E:polynomial}
\lim_{N\to\infty} \E_{\bn\in [L_N]^r} \E_{m\in [N]}\,  f(m)\, f(m+p_1(\bn)) \cdots f(m+p_\ell(\bn))=0,
\end{equation}
where $f$ is the M\"obius or the Liouville function, $L_N$ is of the order  $N^{1/d}$ where $d$ is the maximum degree of the polynomials $p_i$,
and the polynomials $p_i$ have the property that $p_i-p_j$ have degree exactly $d$ for all $i\neq j$ (some other cases were handled as well, see \cite[Theorem~1.6]{TZ16b}). If one combines this with
the fact that aperiodic multiplicative functions are Gowers uniform of arbitrary order \cite[Theorem~2.5]{FH14}, it
allows to extend this result to the class of all aperiodic multiplicative functions.

Extending \eqref{E:main0} or \eqref{E:polynomial}   to general families of polynomials with pairwise non-constant differences, even when $f$ is the Liouville function and $\ell=2, r=1$,  remains a challenge.

\subsection{Notation and conventions}
 For reader's convenience, we gather here some notation that we use throughout the article.
  We denote by $\N$ the set of positive integers and by $\P$  the set of prime numbers.
 For  $N\in\N$  we let  $[N]:=\{1,\dots,N\}$. If $A$ is a finite non-empty set we let $\E_{a\in A}:=\frac{1}{|A|}\sum_{a\in A}$.
 We let  $e(t):=e^{2\pi it}$.
We use the letter $f$ to denote  multiplicative functions and the letter $\chi$ to denote   Dirichlet characters.

\section{Key ingredients from number theory and ergodic theory}
In this section we gather the key ingredients from number theory and ergodic theory needed in the proof
of our main results.
\subsection{Averages over typical short intervals}
The next result from  \cite[Theorem A.1]{MRT15} is the key ingredient  in the proof of Theorem~\ref{T:MRT} below,
which in turn is needed in the proof of Theorem~\ref{T:main}.
\begin{theorem}\label{T:RT}
Let $f\in \CM$ be a  multiplicative function such that $M(f;N)\to \infty$ as $N\to \infty$.
Then
$$
\lim_{N\to \infty} \limsup_{M\to \infty}\E_{m\in [M]} |\E_{n\in [N]} f(m+n)|=0.
$$
\end{theorem}
\begin{remarks}
The statement in  \cite[Theorem A.1]{MRT15} involves averages of the form $\E_{M\leq m<2M}$, but the two statements are equivalent.
\end{remarks}

\subsection{An orthogonality criterion}
 We will need the following qualitative variant of an  orthogonality criterion
 of  K\'atai:
\begin{theorem}[\cite{K86}]
\label{T:katai}
Let $a\in \ell^\infty(\N)$ be a sequence such that
$$
\lim_{N\to\infty} \E_{n\in [N]} \, a(pn)\cdot \overline{a(qn)}=0
$$
for all primes $p,q$ with $p\neq q$. Then for every multiplicative function $f\in \CM$
we have
$$
\lim_{N\to\infty}\E_{n\in[N]}\,  f (n)\cdot a(n)=0.
$$
\end{theorem}

\subsection{Ergodic theory, invariant and Kronecker factors}\label{SS:ergodic}
   A \emph{measure preserving system}, or simply \emph{a system}, is a quadruple  $(X,\CX,\mu, T)$, where
 $(X,\CX,\mu)$ is  a probability space
  and the transformation  $T\colon X\to X$ is  invertible,  measurable, and measure preserving (that is, $\mu(T^{-1}A)=\mu(A)$ for every $A\in \CX$).

 Throughout,  we denote  with $T^n$, $n\in \N$,  the composition $T\circ  \cdots \circ T$, and for $F\in L^\infty(\mu)$ we denote with $TF$  the function $F\circ T$.

The \emph{invariant factor}  of a  system is defined to be the linear subspace of $L^2(\mu)$ consisting of all $T$-invariant functions, that is, functions satisfying $Tf=f$.
The  \emph{Kronecker factor}  
 of a system is defined to be the closed linear subspace of $L^2(\mu)$ spanned by all $T$-eigenfunctions, that is,
functions satisfying $Tf=e(\alpha) \, f$ for some  $\alpha \in \R$. We will use the following two facts about these factors:

$\bullet$  A function $F\in L^2(\mu)$ is orthogonal to the invariant factor
if and only if
$$
\lim_{N\to \infty} \E_{n\in [N]}\int   T^nF\cdot  \overline{F} \, d\mu=0.
$$
This is a direct consequence of the  mean ergodic theorem.

$\bullet$
The function   $F\in L^2(\mu)$ is orthogonal to the Kronecker  factor
if and only if its  spectral measure  is continuous, which, by Wiener's theorem, happens exactly when
$$
\lim_{N\to \infty}\E_{n\in [N]} \Big|\int   T^nF\cdot  \overline{F} \, d\mu\Big|=0.
$$

\subsection{Ergodic transference principle}\label{SS:transference}
We give a result that allows  to recast convergent  correlation sequences
as ergodic correlation sequences.
\begin{definition}\label{D:correlations}
We say that the   sequences $a_1,\ldots, a_\ell\in \ell^\infty(\Z)$  {\em admit correlations along the sequence of intervals ${\bf M}:=([M_k])_{k\in\N}$} with $M_k\to \infty$, if the limit
 $$
 \lim_{k\to\infty} \E_{m\in [M_k]}\,  b_1(m+n_1)\cdots b_s(m+n_s)
$$
exists for every $s\in \N$,  all integers $n_1,\ldots, n_s$ (not necessarily distinct), and all sequences
 $b_1,\ldots, b_s$  that
belong to the set $\{a_1,\ldots, a_\ell, \bar{a}_1,\ldots, \bar{a}_\ell\}$.\footnote{Then the limit $
 \lim_{k\to\infty} \E_{m\in [M_k]}\,  F(b_1(m+n_1),\ldots, b_s(m+n_s))$ exists for every $s\in \N$, $n_1,\ldots, n_s\in \Z$, continuous function $F\colon D^s\to \C$ where
 $D$ is the closed complex disc with radius $\max_{j=1,\ldots, \ell}\norm{a_j}_\infty$, and $b_j\in \{a_1,\ldots, a_\ell\}$ for $j=1,\ldots, s$.
 This follows from the Stone-Weierstrass theorem. }
\end{definition}
\begin{remark}
  If $a_1,\ldots, a_\ell\in \ell^\infty(\Z)$, then using a diagonal argument we get that any sequence $(M_k)_{k\in \N}$ of integers with $M_k\to\infty$
has a subsequence $(M_k')_{k\in\N}$, such that the sequences  $a_1,\ldots, a_\ell$ admit correlations along the  intervals $([M'_k])_{k\in\N}$.
\end{remark}

We use the   following transference principle
which can be thought of as a variant of Furstenberg's correspondence principle for sequences:
\begin{proposition}\label{P:correspondence}
 Suppose that the sequences
 $a_0,\ldots, a_\ell\in \ell^\infty(\Z)$  admit
correlations along
the sequence of intervals ${\bf M}:=([M_k])_{k\in\N}$ with $M_k\to\infty$.
Then there exist a  system $(X,\mathcal{X},\mu, T)$ and
functions $F_0,\ldots,F_\ell\in L^\infty(\mu)$, such that
 $$
\E_{m\in {\bf M} }\, b_0(m)\,   b_1(m+n_1)\cdots b_s(m+n_s)= \int   \tilde{F}_0 \cdot T^{n_1}
\tilde{F}_1\cdots T^{n_s}
\tilde{F}_s\, d\mu
$$
for every $s\in \N$,  $n_1, \ldots, n_s\in \Z$, 
where for $j=0,\ldots, s$ the sequence
$b_j$ is equal to either  $a_i$ or $\overline{a_i}$ for some $i\in \{0,\ldots,\ell\}$
and $\tilde{F}_j$ is equal to $F_i$ or $\overline{F_i}$ respectively.
\end{proposition}
\begin{proof}
We let
 $D$ be the closed complex disc with radius $\max_{j=0,\ldots, \ell}\norm{a_j}_\infty$, $Y:=D^{\Z}$, and $X:=Y^{\ell+1}$.
  We equip $Y$ and $X$ with the product topology and let  $\CX$  be the Borel $\sigma$-algebra
of $X$. We let $T_{sh}\colon Y\to Y$ and $T\colon X\to X$ be defined by  $(T_{sh}x)(m):=x(m+1),$ $m\in \Z$,
and
$$
T(x_0, \ldots, x_\ell):=(T_{sh}x_0 \ldots, T_{sh}x_\ell)
$$
where $x_0,\ldots, x_\ell\in Y$.
Also, for $j=0,\ldots, \ell$ we define the functions $F_j\in C(X)$ by
$$F_j(x_0,\ldots, x_\ell):=x_j(0).$$
  Finally, we let
   $$\omega:=(a_0,\ldots, a_\ell)\in X
   $$
  and let
$\mu$  be  the  weak$^*$ limit (which exists by our assumptions)  of the sequence of probability measures
$$
\mu_k:=\E_{m\in [M_k]}\delta_{T^m\omega},\quad k\in \N.
$$

Note $T\in C(X)$,   the measure $\mu$ is $T$-invariant,  and  for $j=0,\ldots, \ell,$  we have
$$
F_j(T^n\omega)=a_j(n), \quad n\in \Z.
$$ Hence,
assuming that  $b_j$ and $\tilde{F}_j$, $j=0,\ldots, \ell$,  are as in the statement of the result,  letting $n_0=0$, we have   for all   $n_1,\dots,n_s\in\Z$ that
 $$
\E_{m\in {\bf M} }\prod_{j=0}^s
b_j(m+n_j)=
\E_{m\in {\bf M} }\prod_{j=0}^s
T^{n_j}\tilde{F}_j(T^m\omega)
=\int \prod_{j=0}^s T^{n_j}\tilde{F}_j\,d\mu.
$$
This completes the proof.
\end{proof}

\subsection{Two multiple ergodic theorems}
We will use the following result from ergodic theory:
\begin{theorem}[\cite{FK06} for $r=1$]\label{T:PolyErgodic}
Let $p_1,\ldots, p_\ell \colon \N^r\to \Z$ be  independent polynomials, $(X,\CX, \mu,T)$ be a system, and $F_0, \ldots, F_\ell\in L^\infty(\mu)$ be functions at least one of which  is orthogonal to the Kronecker factor of the system. Then
$$
\UD_{\bn\to\infty}
\Big(\int F_0\cdot T^{p_1(\bn)}F_1 \cdots T^{p_\ell(\bn)}F_\ell\, d\mu\Big)=0.
$$
\end{theorem}
This result was proved for $r=1$ in \cite[Lemma~4.3]{FK06}. The proof was based on
the theory of characteristic factors  \cite{HK05, L05c} and on qualitative
equidistribution results on nilmanifolds from \cite{L05a}. An intermediate result
is the convergence result of \cite{FK05} which can be proved for general $r\in \N$
in exactly the same way using equidistribution results for multi-variable sequences  from \cite{L05b} in place of the equidistribution results for single-variable sequences from \cite{L05a}.
For general $r\in \N$, the remaining  part of the proof  is essentially identical to the one given in
 \cite[Lemma~4.3]{FK06} for $r=1$.

\begin{theorem}[\cite{Fr10}]\label{T:HardyErgodic}
Let  $c_1,\ldots, c_\ell\in \R^+$ be distinct non-integers,
$(X,\CX, \mu,T)$ be a system,  and $F_0,$ $\ldots,$ $F_\ell\in L^\infty(\mu)$ be functions.
\begin{enumerate}
\item If  one of the functions is orthogonal to the invariant factor of the system, then
$$
\lim_{N\to \infty}\E_{n\in [N]}
 \int  F_0 \cdot  T^{[n^{c_1}]}F_1\cdots  T^{[n^{c_\ell}]}F_\ell\, d\mu =0.
$$

\item If  one of the functions is orthogonal to the Kronecker factor of the system, then
$$
\lim_{N\to \infty}\E_{n\in [N]}
 \Big|\int  F_0 \cdot  T^{[n^{c_1}]}F_1\cdots  T^{[n^{c_\ell}]}F_\ell\, d\mu \Big|=0.
$$
\end{enumerate}
\end{theorem}
The first part is a direct consequence of \cite[Theorem~2.6]{Fr10}. The second part follows by applying the first part
for the product system $(X\times X, \mu\times \mu,T\times T)$ and the functions $F_j\otimes\overline{F_j}$, $j=0,\ldots, \ell$,
and using the well known fact that if $F$ is orthogonal to the Kronecker factor of a system, then $F\otimes \overline{F}$ is
orthogonal to the invariant factor of the product system.

\section{A consequence of strong aperiodicity}

The next result records a key asymptotic property  of single correlations
of strongly aperiodic multiplicative functions.
\begin{theorem}[\cite{MRT15}]\label{T:MRT}
Let $f\in \CM$ be a strongly  aperiodic multiplicative function that
admits
correlations along the  sequence of intervals ${\bf M}:=([M_k])_{k\in\N}$ with $M_k\to \infty$.
Then
$$
\lim_{N\to \infty}\E_{n\in [N]} |\E_{m\in {\bf M}} \, f(m+n)\cdot \overline{f(m)} |=0.
$$
\end{theorem}
\begin{remark}
  It follows from \cite[Theorem~B.1]{MRT15} that strong aperiodicity cannot be replaced by aperiodicity; in particular,  there exist  an aperiodic multiplicative function $f$, a positive constant $c$,  and a sequence of intervals ${\bf M}:=([M_k])_{k\in\N}$ with $M_k\to \infty$, such that
$$
|\E_{m\in {\bf M}} \, f(m+n)\cdot \overline{f(m)} |\geq c, \quad \text{ for every } n\in \N.
$$
\end{remark}
 A quantitative variant  of Theorem~\ref{T:MRT}
is implicit in \cite{MRT15}. Building on Theorem~\ref{T:RT}, we give in this section a self contained and rather simple proof of Theorem~\ref{T:MRT}.

 We start with the following ``inverse theorem'':
\begin{lemma}\label{L:inverse}
  Let  $a\in \ell^\infty(\N)$ be a sequence that admits  correlations along the sequence of intervals
  ${\bf M}:=([M_k])_{k\in\N}$ and suppose that
$$
\limsup_{N\to \infty} \E_{n\in [N]}|\E_{m\in {\bf M} } \, a(m+n)\cdot  \overline{a(m)} |>0.
$$
Then there exists $t\in \R$ such that
\begin{equation}\label{E:orthogonality}
\limsup_{N\to \infty} \limsup_{M\to \infty}\E_{m\in [M]} |\E_{n\in [N]} \, a(m+n) \cdot e( nt)|>0.
\end{equation}
\end{lemma}
\begin{proof}
Using the Cauchy Schwarz inequality, our assumption gives that
$$
\limsup_{N\to \infty}\E_{n\in [N]}\big(\E_{m\in {\bf M}}\, a(m+n)\cdot \overline{a(m)} \cdot c(n)\big)>0
$$
where
$$
c(n):=\E_{m\in {\bf M} }\, \overline{a(m+n)}\cdot a(m).
$$
Hence,
\begin{equation}\label{E:positive}
\limsup_{N\to\infty}\limsup_{k\to  \infty}\E_{m\in [M_k]}\,| \E_{n\in [N]} a(m+n)\cdot c(n)|>0.
\end{equation}

The sequence $(c(n))$ is positive definite, hence, by Herglotz's theorem we have
$$
c(n)=\int e(nt)\, d\sigma(t), \quad n\in \N,
$$ for some finite positive  measure $\sigma$ on $\T$. We decompose
$\sigma$ to its discrete and continuous  part, and deduce that $c(n)$ can be expressed in the form
$$
c(n)=\sum_{k=1}^\infty c_k \, e(nt_k)+d(n)
$$
for some $t_k\in \T$ and $c_k\in \R^+$, $k\in \N$,  such that   $\sum_{k=1}^\infty c_k<\infty$, and some
sequence $(d(n))$ which, by Wiener's theorem,  converges to $0$ in uniform density.

We deduce form this and  \eqref{E:positive} that there exists $t\in \R$ such that
$$
\limsup_{N\to\infty}\limsup_{k\to  \infty}\E_{m\in [M_k]}| \E_{n\in [N]} \, a(m+n)\cdot e(nt)|>0
$$
which implies the asserted claim.
\end{proof}

Thus, in order to prove Theorem~\ref{T:MRT} it suffices to prove the following result:
\begin{proposition}
Let $f\in \CM$  be a strongly  aperiodic multiplicative function.
Then
$$
\lim_{N\to \infty} \limsup_{M\to \infty}\E_{m\in [M]} |\E_{n\in [N]}\,  f(m+n) \cdot e(nt)|=0
\quad \text{ for every } t\in \R.
$$
\end{proposition}
\begin{proof}
The asserted claim is equivalent to
\begin{equation}\label{E:t}
\lim_{N\to \infty} \limsup_{M\to \infty}\E_{m\in [M]} |\E_{n\in [N]} \, f(m+n) \cdot e((m+n)t)|=0
\quad \text{ for every } t\in \R.
\end{equation}

Suppose first that $t\in \Q$. Then using a standard argument it suffices to show that
\begin{equation}\label{E:rat}
\lim_{N\to \infty} \limsup_{M\to \infty}\E_{m\in [M]} |\E_{n\in [N]} \, f(m+n) \cdot \chi(m+n)|=0
\end{equation}
for every Dirichlet character  $\chi$. Since $f$ is strongly  aperiodic, we have that $f\cdot \chi$ is also
strongly aperiodic; hence, \eqref{E:rat}
is a direct consequence of Theorem~\ref{T:RT} applied to $f\cdot \chi$.

Suppose now that $t$ is irrational. In this case, we will show that \eqref{E:t} holds for every multiplicative function $f\in \CM$.
 It is not hard to check (for details see the proof of Theorem~5 in \cite{ALR15})  that in place of \eqref{E:t}   it suffices to show the following: If $f\in \CM$, $t$ is irrational, and $(b_k)_{k\in \N}$ is an increasing sequence of positive integers with $b_{k+1}-b_k\to \infty$ as $k\to \infty$, then
$$
\lim_{K\to \infty} \frac{1}{b_K}\sum_{k\in [K]}\Big|\sum_{n\in [b_k,b_{k+1})} f(n) \cdot e(nt)\Big|=0.\footnote{This identity
  was also proved in \cite[Theorem~4]{ALR15} using a disjointness argument for topological models of irrational rotations. We give a simpler direct  argument.}
$$
Equivalently, it suffices to show that for every sequence of complex numbers $(c_k)_{k\in \N}$ of modulus $1$
we have
$$
\lim_{K\to \infty} \frac{1}{b_{K+1}}\sum_{k\in [K]}\, c_k \, \sum_{n\in [b_k,b_{k+1})}  f(n) \cdot e(nt)=0,
$$
or equivalently, that
$$
\lim_{K\to \infty} \frac{1}{b_K}\sum_{n\in [b_K]}f(n) \cdot a(n) =0,
$$
where
$$
a(n):=\sum_{k=1}^\infty  c_k\, \one_{[b_k,b_{k+1})}(n) \cdot e(nt).
$$
By Theorem~\ref{T:katai} it suffices to show that for all distinct primes $p,q$ we have
$$
\lim_{N\to \infty} \E_{n\in [N]}\, a(pn)\cdot \overline{a(qn)} =0,
$$
or equivalently, that
\begin{equation}\label{E:C}
\lim_{N\to \infty} \E_{n\in [N]} \, e(n s) \cdot C(n)=0,
\end{equation}
where $s:=(p-q)t$ is irrational and for $k,k'\in \N$ we let
$$
I_{k,k'}:=[b_k/p,b_{k+1}/p)\cap [b_{k'}/q,b_{k'+1}/q)
$$ and
$$
C(n):=\sum_{k,k'\in \N}
 \one_{I_{k,k'}}(n)\cdot c_k\cdot \overline{c_{k'}}.
$$
Note that the sequence of disjoint intervals $(I_{k,k'})_{k,k'\in \N}$ partitions
$\N$ into a set of density $0$ and a sequence of intervals $(J_l)_{l\in \N}$ with $|J_l|\to \infty$ as $l\to \infty$.
Therefore, letting
$$
C'(n):=\sum_{l=1}^\infty
 \one_{J_{l}}(n)\cdot d_l,
$$
where for $l\in \N$ the complex number $d_l$ is equal to $c_k\cdot c_k'$ for  appropriate $k,k'\in \N$,
we have
\begin{equation}\label{E:CC'}
\lim_{N\to \infty}\E_{n\in [N]}|C(n)-C'(n)|=0.
\end{equation}
 Since $s$ is irrational and $|J_l|\to \infty$, we have
$$
\lim_{l\to \infty} \E_{n\in [J_l]} \, e(n s)=0,
$$
hence
$$
\lim_{N\to \infty} \E_{n\in [N]} \, e(n s) \cdot C'(n)=0.
$$
This, combined with \eqref{E:CC'}, implies \eqref{E:C}, and completes the proof.
\end{proof}

\section{Averaged Chowla-Elliott conjecture along independent sequences}
In this section we prove   Theorems~\ref{T:main} and \ref{T:Hardy}. Our approach is to translate
the number theoretic statements to ergodic ones which we then  verify using the ergodic Theorems~\ref{T:PolyErgodic} and \ref{T:HardyErgodic}.
In order to show that the hypothesis of the ergodic theorems are satisfied  we
 use the  number theoretic feedback of Theorems~\ref{T:RT} and \ref{T:MRT}.

\medskip

\noindent {\bf Working Assumptions.} Throughout this section, we assume that a given  collection  $\{f_0,\ldots, f_\ell\}$ of multiplicative functions in $\CM$
 admits correlations along the sequence of intervals ${\bf M}:=([M_k])_{k\in\N}$ (such a sequence always exists),
 and that the measure preserving system
  $(X,\mathcal{X},\mu, T)$ and the functions  $F_0,\ldots,F_\ell\in L^\infty(\mu)$ are defined as in Proposition~\ref{P:correspondence} with $f_0,\ldots, f_\ell$ in place of the sequences $a_0,\ldots, a_\ell$.

\subsection{Orthogonality to the invariant and the Kronecker factor}
The following key facts are  direct consequences of the number theoretic results stated in Theorems~\ref{T:RT} and \ref{T:MRT}:
\begin{proposition}\label{P:key}
   Let  $f_0,\ldots, f_\ell\in \CM$ be multiplicative functions and let  the system  $(X, \CX,\mu,T)$ and the functions $F_0,\ldots, F_\ell\in L^\infty(\mu)$ be as in our working assumptions. We have the following:
\begin{enumerate}
\item If $f_j$ satisfies $M(f_j;N)\to \infty$ as $N\to\infty$ for some $j\in \{0,\ldots, \ell\}$, then $F_j$ is orthogonal to the invariant factor of the system.

\item  If $f_j$ is strongly aperiodic for some $j\in \{0,\ldots, \ell\}$, then  $F_j$ is orthogonal to the Kronecker factor of the system.
\end{enumerate}
\end{proposition}
\begin{remark}
We deduce that if $f_j$ is the M\"obius or the Liouville function, then the function $F_j$ is
 orthogonal to the Kronecker factor of the corresponding system.
\end{remark}
\begin{proof}

$(i)$ By Theorem~\ref{T:RT} we have
$$
\lim_{N\to \infty} \limsup_{k\to\infty}\E_{m\in  [M_k]} |\E_{n\in [N]} f_j(m+n)|=0
$$
which implies that
$$
\lim_{N\to \infty} \E_{n\in [N]}  \E_{m\in {\bf M}} f_j(m+n) \cdot \overline{f_j(m)}=0.
$$
Using our  working assumptions, we deduce that
$$
\lim_{N\to \infty}\E_{n\in [N]} \int   T^nF_j \cdot  \overline{F_j} \, d\mu =0.
$$
As mentioned in Section~\ref{SS:ergodic},
this implies that the function
$F_j$ is orthogonal to the invariant factor of the system.

$(ii)$ By Theorem~\ref{T:MRT} we have
$$
\lim_{N\to \infty}\E_{n\in [N]} |\E_{m\in {\bf M}} \,f_j(m+n)\cdot \overline{f_j(m)} |=0.
$$
Using our  working assumptions, we deduce that
$$
\lim_{N\to \infty}\E_{n\in [N]} \Big|\int   T^nF_j \cdot  \overline{F_j} \, d\mu\Big|=0.
$$
As mentioned in Section~\ref{SS:ergodic},
this implies that the function $F_j$ is orthogonal to the Kronecker factor of the system.
\end{proof}

\subsection{Proof of Theorems~\ref{T:main} and ~\ref{T:Hardy}}

We can now prove our main results.
\begin{proof}[Proof of Theorems~\ref{T:main}]
Let the system  $(X, \CX,\mu,T)$ and the functions $F_0,\ldots, F_\ell\in L^\infty(\mu)$ be as in our  working assumptions.
Then for every $\bn\in \N^r$ we have that
$$
\E_{m\in {\bf M}}\,  f_0(m) \, f_1(m+p_1(\bn))\cdots f_\ell(m+p_\ell(\bn))=\int F_0\cdot T^{p_1(\bn)}F_1\cdots  T^{p_\ell(\bn)}F_\ell\, d\mu.
$$
Hence, in order to verify \eqref{E:main0} it suffices to show that
\begin{equation}\label{E:main0'}
\UD_{\bn\to\infty}
\Big(\int F_0\cdot T^{p_1(\bn)}F_1\cdots  T^{p_\ell(\bn)}F_\ell\, d\mu\Big)=0.
\end{equation}
 Using our assumption and part $(ii)$ of Proposition~\ref{P:key}, we get that one of the functions $F_0,\ldots, F_\ell$  is orthogonal to the Kronecker factor of the system$(X, \CX,\mu,T)$.
Hence,  Theorem~\ref{T:PolyErgodic} implies \eqref{E:main0'}, and  completes the proof.
\end{proof}

In a similar fashion we prove Theorem~\ref{T:Hardy}.
\begin{proof}[Proof of Theorem~\ref{T:Hardy}]
We argue as in the proof of Theorem~\ref{T:main}, using part $(i)$ and  $(ii)$ of  Proposition~\ref{P:key} and
part $(i)$ and $(ii)$ of Theorem~\ref{T:HardyErgodic}.
\end{proof}

\section{Patterns in arithmetic sequences}
In this section we prove Theorems~\ref{T:sign1} and \ref{T:sign2}.
We first establish a preliminary result needed in the proof of
Theorem~\ref{T:sign2}.
\subsection{A class of strongly aperiodic multiplicative functions}
We will use the following  basic facts about $\D$ (see Definition~\ref{D:D}):
For every $N\in \N$ and $f,g,h\in \CM$ we have
(see for example \cite{GS16})
\begin{equation}\label{E:triangle1}
\D(f,g; N)\leq \D(f,h; N)+\D(h,g;N).
\end{equation}
Also for every $N\in \N$ and all $f_1,f_2,g_1,g_2\in \CM$ we have (see \cite[Lemma~3.1]{GS07})
\begin{equation}\label{E:triangle2}
\D(f_1f_2,g_1g_2; N)\leq \D(f_1,g_1; N)+\D(f_2,g_2; N).
\end{equation}
We will also use two facts about $\D(1,n^{it}; N)$. The first can be found in the proof of
\cite[Lemma C.1]{MRT15}. 
It states that for every  $N\geq 100$ we have
\begin{equation}\label{E:basicest1}
 \D(1,n^{it}; N)=\sqrt{\log(1+|t|\log N)}+O(1), \ \text{ if } |t|\leq 1.
\end{equation}
The second follows from an  estimate proved in \cite[Lemma~2]{MR15} (and is attributed to Granville and Soundararajan)
using Vinogradov's zero free region for the Riemann zeta function. It states that
for every $A>0$ we have
\begin{equation}\label{E:basicest2}
\min_{1\leq |t|\leq  N^A} \D(1,n^{it}; N)\to \infty \quad \text{ as } \ N\to \infty.
\end{equation}

The next criterion enables us to  prove that certain multiplicative functions used in the proof of Theorem~\ref{T:sign2} are strongly aperiodic:
\begin{proposition}\label{P:aperiodic}
Let  $d\in \N$ and    $f\in \CM$ be a multiplicative function
such that $f(p)$ is a $d$-th root of unity
 for all but finitely many primes $p$. Suppose that $\D(f, \chi)=\infty$ for every Dirichlet character $\chi$.
 Then $f$ is strongly aperiodic.
 \end{proposition}
\begin{proof}
Let $\chi$ be a Dirichlet character.
There exists $k\in \N$ such that  $(\chi(p))^k=1$ for all but a finite number of primes $p$. Then
for $m=dk$,   using \eqref{E:triangle2} we get  that
$$
m\, \D(f\cdot \chi, n^{it};N)\geq \D(f^m\cdot \chi^m, n^{imt};N)=\D(1, n^{imt};N)+O(1).
$$
Combining this with \eqref{E:basicest2} we deduce that
\begin{equation}\label{E:tgeq1}
\min_{1\leq |t|\leq N} \D(f\cdot \chi, n^{it};N)\to \infty  \quad \text{ as } \ N\to \infty.
\end{equation}

Next, we get lower bounds for $\D(f\cdot \chi, n^{it};N)$ for $|t|\leq 1$.
Using \eqref{E:triangle1} we have  that
$$
\D(f\cdot \chi, n^{it};N)\geq \D(f\cdot \chi, 1;N)-\D(1, n^{it};N)\geq \frac{1}{2}\, \D(f\cdot \chi, 1;N)
$$
unless
 \begin{multline*}
\frac{1}{2}\, \D(f\cdot \chi, 1;N)\leq \D(1, n^{it};N)=\D(1,n^{imt}; N)+O(1)\\
=\D(f^m\cdot \chi^m,n^{imt}; N)+O(1)
\leq
m\, \D(f\cdot\chi,n^{it}; N)+O(1),
\end{multline*}
where the first identity follows
 for $|t|\leq 1$ from \eqref{E:basicest1} and the last estimate from \eqref{E:triangle2}.
In either case, we have
\begin{equation}\label{E:tleq1}
\min_{|t|\leq 1}\D(f\cdot\chi,n^{it}; N)\geq \frac{1}{2m}\, \D(f, \overline{\chi};N)+ O(1).
\end{equation}

Combining  \eqref{E:tgeq1} and \eqref{E:tleq1}, with our assumption $\D(f, \overline{\chi})=\infty$,  we  deduce that
$$
\min_{|t|\leq N} \D(f\cdot \chi,n^{it};N)\to \infty \ \text{ as }\   N\to \infty.
$$
 Hence, $f$ is strongly aperiodic.
\end{proof}

\begin{corollary}\label{C:aperiodic}
Let  $d\in \N$ and    $f\in \CM$ be a multiplicative function
such that $f(p)$ is a non-trivial $d$-th root of unity
 for all but finitely many primes $p$.
 Then $f$ is strongly aperiodic.
 \end{corollary}
 \begin{proof}
Using Proposition~\ref{P:aperiodic} it suffices to show that $\D(f,\chi)=\infty$ for every Dirichlet character $\chi$. Suppose that $\chi$ has period $m$.
Since $\chi(1)=1$, we have  $\chi(n)=1$ whenever $n\equiv 1 \! \! \! \mod{m}$, and since $f(p)$ is a non-trivial $d$-th root of unity  for all but finitely many primes $p$, we have
$$
\D(f, \chi)^2\geq \Big(1-\cos\big(\frac{2\pi}{d} \big)\Big) \cdot\sum_{p\in \P\cap (m\Z+1)}\frac{1}{p}+ O(1)=\infty,
$$
where the divergence of the last series follows from Dirichlet's theorem. This completes the proof.
 \end{proof}
\subsection{Proof of   Theorem~\ref{T:sign1}  }
Let  $(M_k)_{k\in\N}$ be a sequence of positive integers with $M_k\to \infty$
 so that
  the multiplicative functions $f_0,\ldots, f_\ell\in \CM$  admit correlations  along the sequence of   intervals ${\bf M}=([M_k])_{k\in\N}$.

Note that since for $j=0,\ldots, \ell$ the multiplicative function $f_j$ takes values in $\{-1,+1\}$ and $\epsilon_j\in \{-1,+1\}$, we have
$$
{\bf 1}_{f_j=\epsilon_j}(n)=\frac{1+ \epsilon_j f_j(n)}{2}, \quad n\in\N.
$$
Hence,
 $$
 d_{\bf M}(\Lambda_{ \bn, {\bf \epsilon}})= \frac{1}{2^{\ell+1}}\, \E_{m\in {\bf M}}\prod_{j=0}^\ell
 \big(1+\epsilon_jf_j(m+p_j(\bn))\big)
 $$
where $p_0:=0$. Expanding the product to $2^{\ell+1}$ terms and using Theorem~\ref{T:main}, we deduce that $
\UD_{\bn\to\infty}(d_{\bf M}(\Lambda_{\bn,{\bf \epsilon}}))=2^{-(\ell+1)},
$
 completing the proof of  Theorem~\ref{T:sign1}.
\subsection{Proof of   Theorem~\ref{T:sign2}  }
Let $b\in \N$    and $a\in \{0,\ldots, b-1\}$.
We first express the indicator $\one_{ [\omega]_{b}=a}$ as a weighted average of  multiplicative functions.
Let $\zeta$ be a root of unity of order $b$.
We define the multiplicative function
$f_{b}$  by
$$
f_{b}(p^j):=\zeta, \quad j\in \N, \ \ p\in \P.
$$
Then
\begin{equation}\label{E:multident}
\one_{ [\omega]_{b}=a}(n)=  \frac{1}{b}\sum_{r=0}^{b-1} \zeta^{-ar}(f_b(n))^r, \quad n\in\N.
\end{equation}

Let  $(M_k)_{k\in\N}$ be a sequence of positive integers with $M_k\to \infty$
 such that
  the multiplicative functions
  $f_{b_0},\ldots, f_{b_0}^{b_0-1}, \ldots, f_{b_\ell},\ldots, f_{b_\ell}^{b_\ell-1}$
  admit correlations  along the sequence of   intervals ${\bf M}=([M_k])_{k\in\N}$.
Using \eqref{E:multident} we get that
 \begin{equation}\label{E:multiformula}
 d_{\bf M}(\Lambda_{ \bn, {\bf a}})= \E_{m\in {\bf M}}\prod_{j=0}^\ell
 \Big(\frac{1}{b_j}\sum_{r=0}^{b_j-1} \zeta^{-ar}(f_{b_j}(m+p_j(\bn)))^r\Big)
 \end{equation}
where $p_0:=0$.

 Note that by Corollary~\ref{C:aperiodic}, for $j=0,\ldots, \ell$ the multiplicative functions  $f_{b_j}^r$ are strongly aperiodic  for  $r=1,\ldots, b_j-1$.
  Expanding the product in \eqref{E:multiformula} to  $\prod_{j=0}^\ell b_j$ terms and using Theorem~\ref{T:main}, we deduce that
$
\UD_{\bn\to\infty}(d_{\bf M}(\Lambda_{\bn,{\bf a}}))=(\prod_{j=0}^\ell b_j)^{-1},$
  completing the proof of  Theorem~\ref{T:sign2}.

 In a similar fashion we can prove a modification of Theorem~\ref{T:sign2}
 for the arithmetic function $\Omega$ in place of $\omega$;
 the only difference is that we use in place of the multiplicative function $f_b$ the
  completely multiplicative function $f'_b$ defined  by $f'_b(p^j):=\zeta^j$,  for all  $j\in \N$ and primes $p$.

\end{document}